\documentclass[12pt,psamsfonts]{amsart}

\usepackage{amssymb,amsfonts,amsthm,amsmath,relsize}

\theoremstyle{plain}
\newtheorem{theorem}{Theorem}

\newtheorem{proposition}{Proposition}
\newtheorem{corollary}{Corollary}

\theoremstyle{definition}
\newtheorem{definition}{Definition}

\newtheorem{remark}{Remark}

\def\Z{\mathbb Z}
\def\s{\textstyle{\frac{p-1}2}}
\def\sn{\textstyle{\frac{n-1}2}}
\def\sf{\operatorname{sf}}      
\def\H{\operatorname{H}}

\title[Double-, Hyper-, Sub- and Super-factorials]{Wilson Theorems for \\Double-, Hyper-, Sub- and Super-factorials}

\author{Christian Aebi}
\author{Grant Cairns}

\address{Coll\`ege Calvin, Geneva, Switzerland 1211}
\email{christian.aebi@edu.ge.ch}
\address{Department of Mathematics and Statistics, La Trobe University, Melbourne, Australia 3086}
\email{G.Cairns@latrobe.edu.au}

\begin{document}

\maketitle

\begin{abstract}
We present generalisations of Wilson's theorem for double factorials, hyperfactorials, subfactorials and superfactorials.
\end{abstract}

\section{Introduction}

Wilson's theorem states that  $(p-1)!\equiv -1$ if $p$ is prime, and $(p-1)!\equiv 0$ otherwise, except for the one special case, $p=4$. 
 The result is attributed to John Wilson, a student of Waring, but it has apparently been known for over a thousand years; see  \cite{Ra}, \cite[Ch. II]{Dav}, \cite[Ch. 11]{Ore}, \cite[Chap.~3]{Dick1} and \cite[Chap.~3.5]{St}. The first  proof was given by Lagrange \cite{La}. We give generalisations of Wilson's theorem for the four standard generalisations of the factorial function.
 
\begin{definition} For a natural number $n$,
\begin{enumerate}
\item  the \emph{double factorial}  $n!!$ is the product of the natural numbers less than or equal to $n$ that have the same parity as $n$, 
\item  the \emph{hyperfactorial}  $\H(n)$ is the number $\H(n)=\prod_{k=1}^nk^k$, 
\item  the \emph{subfactorial}  $!n$ is  the number of permutations  of the set $\{1,2,\dots,n\}$ that fix no element, 
\item  the \emph{superfactorial}  $\sf(n)$ is  the number  $\sf(n)=\prod_{k=1}^nk!$ 
\end{enumerate}
\end{definition}

The idea underlying our approach is the following: if there are sensible versions of Wilson's Theorem for these functions, then the numbers involved must be very special. If $p$ is prime, then the obvious \emph{special} numbers in the field $\Z_p$ are $0,\pm1$ and when $p$ is congruent to $1$ mod $4$, the two square roots of $-1$. So the task is to verify that these are the values taken, and determine which value is taken in each case. We find that this simple approach for prime $p$ works surprisingly well. For composite numbers, the hyperfactorial, subfactorial and superfactorial functions are all easily treated, while the double factorial holds some unexpected interesting surprises. 
 We present the results for double factorials in Theorems  \ref{T:3mod4}, \ref{T:oddprime}, \ref{T:comp} and \ref{T:even}. The results for the superfactorial, hyperfactorial and subfactorial are given in Theorems \ref{SFDF}, \ref{hyp} and \ref{sub} respectively.  The following result is an unexpected consequence of this study.
 
 \begin{theorem} If $p$ is an odd prime, then modulo $p$
 \[
 (p-1)!!\equiv \sf(p-1)\equiv (-1)^{\frac{p-1}2} \H(p-1).
 \]
 \end{theorem}

\section{Factorials: double, hyper, sub and super}

According to the  MacTutor History of Mathematics archive, the name \emph{factorial}  and the notation $n!$ were introduced by the French mathematician Christian Kramp in 1808 \cite{Kr}, but the symbol was not immediately universally adopted. In the English speaking world, the notation ${\mathsmaller{\lfloor}}\underbar{n}$ was still commonly used  at the end of the $19^{th}$ Century \cite{Caj}. De Morgan wrote ``Among the worst of barbarisms\footnote{At some point, the word barbarisms was misspelt as barabarisms when quoted and this error has been reproduced in a great number of places.} is that of introducing symbols which are quite new in mathematical, but perfectly understood in common, language. Writers have borrowed from the Germans the abbreviation $n!$ to signify $1.2.3\ldots(n - 1).n$, which gives their pages the appearance of expressing surprise and admiration that $2, 3, 4, \&c$.~should be found in mathematical results'' \cite{DeM}.

Rev.~W.~Allen~Whitworth\footnote{
There is an amusing Australian connection concerning Whitworth. In the Cairns Post,  22 November 1890, an article reads: ``The Rev.~W.~Allen Whitworth, in attempting a definition of gambling, as separable from mercantile  speculation and legitimate enterprise, committed himself to the
proposition that under ordinary circumstances
and within the limits of moderation, it is (1)
justifiable to back one's skill, (2) foolish to
back one's luck, (3) immoral or fraudulent
to back one's knowledge. There is no rule
(4), but had there been it would doubtless
have read -- highly commendable to back
Carbine''.  Carbine was the horse that won the Melbourne Cup in 1890.} apparently didn't share De~Morgan's view. 
Whitworth introduced the  \emph{subfactorial} and a symbol for it, in a paper  that begins with the words: ``A new symbol in algebra is only half a benefit unless it has a new name. We believe that the symbol ${\mathsmaller{\lfloor}}\underbar{n}$ as an abbreviation of the continued product of the first $n$ integers, was long in use before the name \emph{factorial} $n$ was adopted. But until it received its name it appealed only to the eye and not to the ear, and in reading aloud could only be described by a periphrasis'' \cite{Whit}. Subfactorial $n$ is the number of permutations  of the set $\{1,2,\dots,n\}$ that fix no element. There are many symbols for the subfactorial. Whitworth used the symbol $\underline{{\mathsmaller{\lfloor}}\underbar{n}}$, in keeping with the notation for the factorial at the time. These days $n\text{!`}$ is sometimes used, but $!n$ seems more common. De~Morgan would not be happy!

The subfactorial has the following explicit formula.
\begin{equation}\label{ELsub}
!n=n!\sum_{i=0}^n \frac{(-1)^i}{i!}.
\end{equation}
The hyperfactorial function was also introduced in the $19^{th}$ Century \cite{Gl,Ki} and the superfactorial function
followed at the turn of the Century \cite{Ba}, although the terms were introduced much later \cite{Sl,GKP}.

Study of the double factorial goes back at least as far as 1948 \cite{Me}, but it is probably much older. 
For even $n$, one has
\begin{equation}\label{E:df}
n!!=2^{\frac{n}2}\cdot  {\textstyle{\frac{n}2}}!
\end{equation}
For $n$ odd, $n!!$  coincides with the Gauss factorial $n_2!$, where  in general $n_m!$ is the product of the natural numbers $i\leq n$ that are relatively prime to $m$ \cite{cd2}.
For a recent paper on the properties of double factorials and their applications, see \cite{GQ}.

\section{Motivation}\label{S:mot}

Our motivation for considering the various factorial functions concerns the matrix
\[
A=
\begin{pmatrix}
1&2&3&\dots &p-1\\
1&2^2&3^2&\dots& (p-1)^2\\
1&2^3&3^3&\dots& (p-1)^3\\
\vdots&\vdots&\vdots&&\vdots\\
1&2^{p-1}&3^{p-1}&\dots& (p-1)^{p-1}
\end{pmatrix}.
\]
There are many interesting facts and open questions related to the properties of the matrix $A$ modulo $p$. For example, \emph{Giuga's conjecture} is that the sum of the entries in the bottom row is congruent to $-1$ if and only if $p$ is prime \cite{BBBG}.

\begin{proposition}\label{L:matrix} The above matrix $A$ has determinant $\sf(p-1)$.
\end{proposition}

\begin{proof} Using the formula for the Vandermonde determinant \cite[Chap.~1.2]{Pr} we have
\begin{align*}
\det(A)&=(p-1)!\ \det
\begin{pmatrix}
1&1&1&\dots &1\\
1&2&3&\dots &p-1\\
1&2^2&3^2&\dots& (p-1)^2\\
\vdots&\vdots&\vdots&&\vdots\\
1&2^{p-2}&3^{p-2}&\dots& (p-1)^{p-2}
\end{pmatrix}\\
&=(p-1)! \prod_{1\leq i < j\leq p-1} (j-i)\\
&=(p-1)! \prod_{2\leq j\leq p-1} (j-1)!=\sf(p-1).
\end{align*}
\end{proof}

\begin{remark}\label{diag}
Notice that $\H(p-1)$ is precisely the product of the elements on the main diagonal of $A$. We will return to the matrix $A$ briefly in Section \ref{Sdf}.
\end{remark}

\section{The connection between the superfactorial and the double factorial}

Somewhat surprisingly, for prime $p$, the values $\sf(p-1)$ and $(p-1)!!$ are congruent.
Before proving this result, let us recall a well known fact which Lagrange proved in \cite{La} (see \cite{GQ,AC1}). Let $p$ be an odd prime. Then 
\begin{equation}\label{R:wil}
\left(\s!\right)^2\equiv(-1)^{\frac{p+1}2}\pmod p.
\end{equation}
Indeed, modulo $p$, one has $p-i\equiv -i$ for all $i$ and so $(p-1)!\equiv(-1)^{\frac{p-1}2}\left(\s!\right)^2$ and the required claim follows from Wilson's Theorem. 

\begin{theorem}\label{SFDF} If $p$ is prime, then $\sf(p-1)\equiv (p-1)!!\pmod p$.
\end{theorem}
\begin{proof} The theorem is obvious for $p=2$.  In the following proof, $p$ is odd and the congruences are taken modulo $p$.
We have
\begin{align}
\sf(p-1)&=(p-1)!\,(p-2)!\,(p-3)!\,\dots\,3!\,2!\,1!\notag\\
&=(p-1)!!\,\left((p-2)!\,(p-4)!\,\dots\,3!\,1!\right)^2.\label{E:dfac}
\end{align}
Note that for all $1\leq i\leq p-1$,
\[
(p-i)! =\frac{(p-1)!}{(p-1)(p-2)\dots(p-i+1)}\equiv \frac{-1}{(-1)(-2)\dots(-i+1)}.
\]
Hence
\begin{equation}\label{E:canc}
(p-i)! \equiv \frac{(-1)^i}{(i-1)!}.
\end{equation}
So if $\frac{p-1}2$ is even,  the factorials  in \eqref{E:dfac} cancel in pairs, giving $\sf(p-1) \equiv (p-1)!!$, as required. If $\frac{p-1}2$ is odd, cancellation of the factorials in \eqref{E:dfac} leaves the middle term, giving $\sf(p-1)  \equiv (p-1)!!\ \left (\s!\right)^2$ and hence $\sf(p-1)  \equiv (p-1)!!$, by \eqref{R:wil}.
\end{proof}

\begin{remark}\label{R:GR}
In the above proof,   the factorials in $\sf(p-1)$ can also be cancelled in pairs using \eqref{E:canc}, leaving only the middle one, and so
\begin{equation}\label{E:gt}
\sf(p-1)\equiv(-1)^{\sum_{i=1}^{\frac{p-1}2}i}
\cdot \s!=(-1)^{\frac{p^2-1}8}\cdot  \s!
\end{equation}
Comparing this with $ (p-1)!! =2^{\frac{p-1}2}\cdot \s!
$, we obtain a simple derivation that $2^\frac{p-1}2\equiv(-1)^{\frac{p^2-1}8}$. In the following we will also require Euler's criterion \cite[Ch. III]{Dav}: 2 is a quadratic residue if and only if $2^\frac{p-1}2\equiv1$.  Using this and  Legendre's symbol we get a basic property of Gaussian reciprocity which allows to write: 
\begin{equation}\label{Q:q}
(p-1)!!\equiv\left( \frac{2}{p}\right)\cdot \frac{p-1}{2}!\equiv (-1)^{\frac{p^2-1}8}\cdot \frac{p-1}{2}! \pmod p.
\end{equation}
%The simplest proof of Euler's criterion uses Wilson's Theorem. See Lemmermeyer's notes http://www.fen.bilkent.edu.tr/~franz/nt/ch6.pdf 
\end{remark}

\section{The double factorial in the prime case}\label{Sdf}

If  $p$ is an odd prime, then  \eqref{R:wil} and \eqref{Q:q}  give 
\begin{equation}\label{extra} 
((p-1)!!)^2\equiv \left(\s!\right)^2
\equiv(-1)^{\frac{p+1}2}  \pmod p.
 \end{equation}
In particular, $(p-1)!! \equiv
 \pm 1\pmod p$ when $p\equiv 3 \pmod{4}$. The following result gives a more precise statement of this fact.

\begin{theorem}\label{T:3mod4}
Suppose that $p$ is an odd prime and  $p\equiv 3 \pmod{4}$. Then
$(p-1)!! \equiv
 (-1)^{\nu}\pmod p$,
where  ${\nu}$ denotes the number of nonquadratic residues $j$ with $2<j<\frac{p}2$. 
\end{theorem}

\begin{proof} Once again, the congruences will be taken modulo $p$ unless otherwise stated. Let $s:=\frac{p-1}{2}$ and $\Z^+_p:= \{1,2,\dots, p-1\}$.  To evaluate the factor $s!$ in  \eqref{Q:q} we use an argument that Mordell attributed to Dirichlet  \cite{Mordell}. Consider the involution $\varphi: x\mapsto p-x$ on $\mathbb \Z^+_p$. Since $p\equiv 3 \pmod{4}$, $p$ cannot be expressed as a sum of two squares, and so  $\varphi$ interchanges  the set $QR$ of quadratic residues with the set $NR$ of nonquadratic residues. Let $QR=\{r_1,r_2,\dots,r_s\}$ and $NR=\{n_1,n_2,\dots,n_s\}$, with elements listed in natural order. Then
\[
s!= r_1r_2\dots r_{s-N}\ n_1n_2\dots n_{N}\equiv (-1)^N  r_1r_2\dots r_s,
\]
where $N$ denotes the number of nonquadratic residues $j<\frac{p}2$. As $p\equiv  3 \pmod 4$, one has $ r_1r_2\dots r_s\equiv 1$; see  \cite[p.~75]{Rose}. Hence $s!\equiv (-1)^N$.
Thus by Remark \ref{R:GR}, $(p-1)!!\equiv 1$ if and only if either $2$ is quadratic residue and $N$ is even, or $2$ is nonquadratic residue and $N$ is odd. Thus $(p-1)!!\equiv 1$ if and only if $\nu$ is 
even.
\end{proof}

When $p\equiv 1 \pmod{4}$, we have $((p-1)!!)^2\equiv -1\pmod p$. In this case, because $-1$ is a quadratic residue, the involution $\varphi$ used in the proof of Theorem \ref{T:3mod4} leaves the sets $QR$ and $NR$ invariant, so it is not useful. Instead we consider the involution $\psi: x\mapsto x^{-1}$ of $\mathbb \Z^+_p$. It turns out that this approach is applicable for all odd primes.

\begin{theorem}\label{T:oddprime}
Suppose that $p$ is an odd prime. Let ${\mu}$ denote the number of elements $j $ less than $\frac{p}2$ such that the inverse $j^{-1}$ of $j$ modulo $p$ is also less than $\frac{p}2$. 

\begin{enumerate}
\item If $p\equiv 3 \pmod{4}$, then
$(p-1)!! \equiv
 (-1)^{\frac{\mu+1}2}\pmod p$.
\item If $p\equiv 1 \pmod{4}$, then
$(p-1)!! \equiv
 (-1)^{\frac{\mu+1}2}i_p\pmod p$,
where  $i_p$ is the unique natural number  less than $\frac{p}2$ with $i_p^2\equiv-1 \pmod p$.
 \end{enumerate}
\end{theorem}

\begin{proof} We use the same notation as in the proof of Theorem \ref{T:3mod4}. 
Notice that if $j <\frac{p}2$ and $ j^{-1} <\frac{p}2$, then the two terms cancel in $\s!$. On the other hand, if $j <\frac{p}2$ and $ j^{-1} >\frac{p}2$, then $p-j^{-1} <\frac{p}2$ and provided $j\not=p-j^{-1}$, the product $j. (p-j^{-1})$ in $\s!$ gives $-1$. 

If $p\equiv 3 \pmod{4}$, the number $-1$ is not a quadratic residue and so  there is no $j <\frac{p}2$ with $j=p-j^{-1}$. In this case, $b=\s!\equiv  (-1)^{\frac{s-\mu}2}$, where $s=\frac{p-1}2$. Hence $(p-1)!!= ab\equiv  (-1)^{\frac{p^2-1}8+\frac{s-\mu}2}$. Let $p=4k+3$. Then $\frac{p^2-1}8+\frac{s}2
=2k^2+4k+ \frac32$. Hence $(p-1)!!\equiv  (-1)^{\frac{3-\mu}2}\equiv  (-1)^{\frac{\mu+1}2}$, as required.

 If $p\equiv 1 \pmod{4}$, we argue in the same manner, but now  $i_p$ is the unique number with  $i_p <\frac{p}2$ and $i_p=p-i_p^{-1}$. 
 Hence $\s!\equiv (-1)^w i_p$, where $w=\frac{s-\mu-1}2$.
Thus $(p-1)!!= ab\equiv  (-1)^{\frac{p^2-1}8+\frac{s-\mu-1}2}i_p$. Let $p=4k+1$. Then $\frac{p^2-1}8+\frac{s}2
=2k^2+2k$. Hence $(p-1)!!\equiv  (-1)^{\frac{\mu+1}2}i_p$, as required.\end{proof}

Together Theorems \ref{T:3mod4} and \ref{T:oddprime}  provide the following equivalence.

\begin{corollary}
If $p$ is an odd prime with $p\equiv 3 \pmod{4}$, then the number $\nu$, of nonquadratic residues $i$ with $2<i<\frac{p}2$, is even if and only if the number $\mu$ of elements $j $ less than $\frac{p}2$ such that the inverse $j^{-1}$ of $j$ modulo $p$ is also less than $\frac{p}2$, is congruent to 3 modulo 4.\end{corollary}

\begin{remark}
The results of Theorems \ref{T:3mod4} and \ref{T:oddprime} can be expressed in terms of class field numbers $h$, but the resulting statements are not as succinct as those given above. For the $p\equiv 3 \pmod{4}$ case, one can use  $h(-p)=2N-1\pmod 4$; see  \cite{Mordell}. For $p\equiv 1 \pmod{4}$, the result can be expressed in terms of $h(p)$ and the fundamental unit of the associated real quadratic number field; see \cite{Ch}.
 \end{remark}

We can now give the connection between the  hyperfactorial and double factorial.

\begin{theorem}\label{hyp} For $p$ an odd prime, the  hyperfactorial and double factorial are connected by the relation
$\H(p-1)\equiv (-1)^{\frac{p-1}2}  (p-1)!!\pmod{p}$. 
\end{theorem}

\begin{proof}
Using Fermat's Little Theorem and Wilson's Theorem, we have 
\begin{align*}
\H(p-1)=\prod_{k=1}^{p-1}k^k&= \frac{(\sf(p-1))^{p-1}}{\sf(p-2)}\\
&= \frac{(\sf(p-1))^{p-1}(p-1)!}{\sf(p-1)}\\
&\equiv \frac{-1}{\sf(p-1)}\equiv \frac{-1}{(p-1)!!},
\end{align*}
by Theorem \ref{SFDF}. By \eqref{extra}, we have 
\[
((p-1)!!)^{-1}\equiv \begin{cases}
(p-1)!!&: \ \text{if}\ p\equiv 3 \pmod{4}\\
-(p-1)!!&: \ \text{if}\ p\equiv 1 \pmod{4}.
\end{cases}
\]
Thus $\H(p-1)\equiv (-1)^{\frac{p-1}2} (p-1)!!$, as required.
\end{proof}

\begin{remark}
Note that by Proposition \ref{L:matrix}, Remark \ref{diag} and Theorems \ref{SFDF} and \ref{hyp},  modulo $p$ the determinant of the matrix $A$ of Section \ref{S:mot} is  $(-1)^{\frac{p-1}2}$ times the product of the elements on the main diagonal of $A$.  \end{remark}

\section{Composite numbers}

When $n$ is composite and $n\not=4$, one has $(n-1)!\equiv 0\pmod n$. Similarly, it is obvious that $\sf(n-1)\equiv 0\pmod n$ and $\H(n-1)\equiv 0\pmod n$ for all composite natural numbers $n$. For the double factorial, the situation is more nuanced. The case of odd composites is not difficult. 

\begin{theorem}\label{T:comp}
If $n$ is a composite odd natural number, then $(n-1)!!\equiv 0 \pmod n$ if $n>9$, while $8!! \equiv 6\pmod9$.\end{theorem}

\begin{proof}
Let $n$ be a composite odd natural number. If $n=ab$, where $a,b$ are co-prime, then $a,b< \frac{n-1}2$, so $\sn!\equiv 0\pmod n$. So we may assume that $n$ is of the form $n=p^k$ for some odd prime $p$, where $k\geq 2$.
If $k>2$, then $p,p^{k-1}$ are distinct and $p,p^{k-1}< \frac{n-1}2$, so $\sn!\equiv 0\pmod n$. If $n=p^2$ and $p>3$, then $p,2p$ are distinct and $p,2p< \frac{n-1}2$, so once again $\sn!\equiv 0\pmod n$. It remains to consider $n=9$, which can be calculated by hand.
\end{proof}

However, when $n$ is even, the pattern is less obvious.

\begin{theorem}\label{T:even}
Suppose that $n=2^iá(2k+1)$, where  $i\geq 1$ and $k\geq 0$.

\begin{enumerate}

\item    if $i=1$, then $(n-1) !! \equiv  2k+1 \pmod{n}$,

\item   if $i=2$, then $(n-1)!! \equiv - (2k+1) \pmod{n}$,

\item  if $i>2$, then $(n-1)!! \equiv (2k+1)^{2^{i-2}} \pmod{n}$. 
\end{enumerate}
\end{theorem}

\begin{proof} We first treat the case $k=0$. Recall that for an arbitrary integer $n$,  Gauss' generalisation of Wilson's Theorem \cite[Chap.~III]{Dick1} (see also \cite{AC2}) states that if $I$ denotes the set of invertible elements in $\Z_n$, then
\[
\prod_{s\in I} s \equiv 
\begin{cases}-1&:\ \text{if}\  n=4,p^\alpha,2p^\alpha\ (p\ \text{an odd prime})\\
1&:\ \text{otherwise}
\end{cases}
\ \pmod{n}.
\]
For $n=2^i$ the set of odd elements of $\Z_{n}$ is precisely the group of invertible elements in $\Z_n$. Hence Gauss' result gives
\[
(n-1)!!\equiv \begin{cases}
-1&:\ \text{if}\ i=2\\
1&:\ \text{otherwise},
\end{cases} \ \pmod{n}.
\]
as required.

Now assume $k>0$. By the Chinese remainder theorem, the map
\begin{align*}
\varphi : \Z_n\to \Z_{2^i}\times \Z_{2k+1}\\
m\mapsto (\varphi_1(m),\varphi_2(m))
\end{align*}
is a ring isomorphism, where $\varphi_1(m)$ (resp.~$\varphi_2(m)$)  is the reduction of $m$ modulo $2^i$ (resp.~$2k+1$). The inverse map is given by
\[
\varphi^{-1}(x,y)\equiv by2^i+ax(2k+1) \pmod n.
\]
where $a(2k+1)+b2^i=1$. The numbers $a$ and $b$ are defined modulo $n$. The proof focuses on the value of $a$. For $i=1$ we may take $a=1$. For $i=2$, we can take $a= (2k+1)$. For $i\geq 3$, note that as the set of odd elements of $\Z_{2^i}$ is the group of invertible elements in $\Z_{2^i}$, so by Euler's Theorem,  $(2k+1)^{2^{i-2}}\equiv 1 \pmod{2^i}$. Thus we may take $a=(2k+1)^{2^{i-3}}$ when $i\geq 3$.
We have
\[
\varphi((n-1)!!)=\prod_{\substack{x\in \Z_{2^i},\ y\in\Z_{2k+1}\\x\ \text{odd}}} (x,y)=\left(\prod_{\substack{x\in \Z_{2^i}\\x\ \text{odd}}} x, \prod_{y\in\Z_{2k+1}} y\right)=\left((2^i-1)!!, 0\right).
\]
Now from the $k=0$ case treated above,
\[
(2^i-1)!!\equiv\begin{cases}
-1&:\ \text{if}\ i=2\\
1&:\ \text{otherwise}.
\end{cases}\ \pmod{2^i}.
\]
So when $i=1$ we have
\[
(n-1)!!=\varphi^{-1}(1,0)\equiv a(2k+1)=(2k+1).
\]
When $i=2$ we have
\[
(n-1)!!=\varphi^{-1}(-1,0)\equiv -a(2k+1)=-(2k+1)^2.
\]
Finally, for $i\geq 3$ we have
\[
(n-1)!!=\varphi^{-1}(1,0)\equiv a(2k+1)=(2k+1)^{2^{i-2}}.
\]
\end{proof}

\section{The subfactorial}

For the (standard) factorial, the double factorial, the hyperfactorial and the superfactorial, the value at number $n$ is congruent to zero mod $n$. It is in part for this reason that the value at $n-1$ is of interest mod $n$. For the subfactorial however, the situation is different. Here the natural generalisation of Wilson's Theorem is the following.

\begin{theorem}\label{sub}
If $n$ is a natural number, then $!n\equiv (-1)^n \pmod n$.
\end{theorem}

\begin{proof} 
Modulo $n$ we have from \eqref{ELsub}
\begin{equation*}
!n=n!\sum_{i=0}^n \frac{(-1)^i}{i!}=(-1)^n+n!\sum_{i=0}^{n-1} \frac{(-1)^i}{i!}\equiv (-1)^n \pmod n.
\end{equation*}
\end{proof}

\bibliographystyle{amsplain}
\bibliography{factorials}

\end{document}